\documentclass[10pt]{article}
\usepackage{amsmath}
\usepackage{amssymb}
\usepackage{amsthm}
\usepackage{stmaryrd}
\usepackage{array}
\usepackage{hyperref}
\usepackage[titletoc,toc,title]{appendix}

\usepackage{tikz}
\usetikzlibrary{matrix,arrows,decorations.pathmorphing}
\usepackage{enumitem}

\usepackage{pdflscape}

\usepackage{caption}

\newtheorem{definition}{Definition}[section]
\newtheorem{theorem}[definition]{Theorem}

\newtheorem{lemma}[definition]{Lemma}

\newtheorem{remark}[definition]{Remark}
\newtheorem{corollary}[definition]{Corollary}

\newtheorem{proposition}[definition]{Proposition}

\def\N{{\mathbb N}}
\def\F{{\mathbb F}}
\def\Z{{\mathbb Z}}

\def\Span{{\rm Span}\ }

\def\PSL2{PSL_2 (\mathbb{Z})}
\def\SL2{SL_2 (\mathbb{Z})}

\def\genset{\mathcal{X}}

\def\monoid{\left<\genset\right>}
\def\lbrack{\left[}
\def\rbrack{\right]}
\def\rev{^*}

\def\freeLie{\mathcal{L}}
\def\ad{{\rm ad}\ }

\def\indset{\mathcal{I}}

\def\indsetL4{\indset^*}

\def\Basis5{\mathcal{B}}
\def\Reps5{\mathcal{S}}

\def\freeassoc3{\F\left< A,B,C\right>}

\def\lnon{\left\llbracket}
\def\rnon{\right\rrbracket}

\def\rtri{\vartriangleright}

\def\algebH{\mathcal{H}}
\def\Heisen{\algebH(q)}

\def\freeH{\algebH(0)}

\def\HLie{\mathfrak{L}}
\def\HeisenLie{\HLie(q)}
\def\BosonLie{\HLie(1)}
\def\freeHLie{\HLie(0)}

\def\charF{{\mbox{char}}\  \F}

\def\BActr{\kappa}

\def\filtrH{\algebH}

\def\LieAB{\lbrack A,B\rbrack}
\def\algI{I}
\def\qHeisenLie{{\overline{\HLie}(q)}}
\def\qHeisen{{\overline{\algebH}(q)}}

\def\baseA{\alpha}
\def\baseB{\beta}
\def\baseG{\gamma}

\def\obaseA{\overline{\baseA}}
\def\obaseB{\overline{\baseB}}
\def\obaseG{\overline{\baseG}}

\def\barQuotient{\qHeisen/\qHeisenLie}



\begin{document}

\title{{\bf Lie polynomials in $q$-deformed Heisenberg algebras}}
\author{{\scshape Rafael Reno S. Cantuba} \\
Mathematics Department\\
De La Salle University Manila\\
{ Taft Ave., Manila, Philippines}\\
{\it rafael{\_}cantuba@dlsu.edu.ph}}

\date{}

\maketitle

\begin{abstract}

Let $\F$ be a field, and let $q\in\F$. The $q$-deformed Heisenberg algebra is the unital associative $\F$-algebra $\Heisen$  with generators $A,B$ and relation $AB-qBA=\algI$, where $\algI$ is the multiplicative identity in $\Heisen$. The set of all Lie polynomials in $A,B$ is the Lie subalgebra $\HeisenLie$ of $\Heisen$ generated by $A,B$. If $q\neq 1$ or the characteristic of $\F$ is not $2$, then the equation $AB-qBA=\algI$ cannot be expressed in terms of Lie algebra operations only, yet this equation still has consequences on the Lie algebra structure of $\HeisenLie$, which we investigate. We show that if $q$ is not a root of unity, then $\HeisenLie$ is a Lie ideal of $\Heisen$, and the resulting quotient Lie algebra is infinite-dimensional and one-step nilpotent.\\

\noindent 2010 Mathematics Subject Classification: 17B60, 16S15, 81R50, 05A30

\noindent Keywords: $q$-deformed Heisenberg algebras; Lie subalgebra; ideal of a Lie algebra
\end{abstract}

\section{Introduction}

The $q$-deformed Heisenberg algebras are said to be algebraic structures with rich properties as mathematical objects and have many important applications in physics and beyond \cite[p. vii]{Hel}. The $q$-deformed Heisenberg algebra $\Heisen$ is an algebraic formalism of the so-called $q$-deformed canonical commutation relation. This is an equation $AB-qBA=\algI$ for some fixed scalar deformation parameter $q$. A common realization is that of $A$ and $B$ being interpreted as mutually adjoint linear operators (with $\algI$ as the identity map) on an infinite-dimensional separable Hilbert space \cite[Section 1.3]{Hor}. Algebraically, we view $\Heisen$ as an associative algebra abstractly defined as having a presentation with generators $A,B$ and relation $AB-qBA=\algI$. 

By a simple routine calculation based on a well-known result in the literature about free Lie algebras \cite[Lemma 1.7]{Reut}, the relation $AB-qBA=\algI$, in general, cannot be expressed in terms of Lie algebra operations only. That is, in general, $AB-qBA-\algI$ is not a \emph{Lie polynomial} in the free unital associative algebra in two generators $A,B$. However, the associative algebra structure in $\Heisen$ allows for the computation of commutators $\lbrack X,Y\rbrack := XY - YX$ for $X,Y\in\Heisen$. Thus, we investigate the interplay of the $q$-deformed and non-deformed commutator structures in $\Heisen$. That is, we investigate the consequences of the deformed commutation relation $AB-qBA=\algI$ on the Lie subalgebra $\HeisenLie$ of $\Heisen$ generated by $A,B$.

A main idea that helps us accomplish our goals in this study is a spanning set for $\HeisenLie$ that is obtained from the theory of free Lie algebras. There is a considerable amount of literature for the basis of a free Lie algebra on a finite number of generators. These arose from seminal works such as \cite{Hall, Shir53, Shir58}. In this work, we use the formulation in \cite[Section 2.8]{Ufn} of the concepts in \cite{Shir53, Shir58}. This deals with so-called \emph{regular words} on the generators. We recall important results and properties about presentations of associative and Lie algebras, about regular associative and nonassociative words, and about the algebra $\Heisen$ in Sections~\ref{PrelSec} and \ref{PropSec}. The case $q=1$ involves the three-dimensional Lie algebra $\BosonLie$ with basis $A,B,\lbrack B,A\rbrack$ that satisfy the commutation relations $\lbrack \lbrack B,A\rbrack,A\rbrack=0=\lbrack B, \lbrack B,A\rbrack\rbrack$. The Lie algebra $\BosonLie$ is well-known. See for instance \cite[Section 1.1]{Hel}. In this work, we consider the case when $q$ is not a root of unity. Our goal is to show that, under the said case, $\HeisenLie$ is an ideal of $\Heisen$ under the Lie algebra structure (i.e., a \emph{Lie ideal}). We further divide this into the case when $q=0$ in Section~\ref{zeroLieSec}, and the case when $q$ is nonzero and not a root of unity, which is in Section~\ref{MainSec}.

\section{Preliminaries}\label{PrelSec}

Let $\F$ be a  field. Throughout, by an $\F$-algebra we mean a unital associative $\F$-algebra, with multiplicative identity $\algI$. Let $\mathfrak{A}$ be an $\F$-algebra. Recall that an anti-automorphism of $\mathfrak{A}$ is a bijective $\F$-linear map $\psi:\mathfrak{A}\rightarrow\mathfrak{A}$ such that $\psi(fg)=\psi(g)\psi(f)$ for all $f,g\in\mathfrak{A}$. We turn $\mathfrak{A}$ into a Lie algebra with Lie bracket $\lbrack f,g\rbrack=fg-gf$ for $f,g\in\mathfrak{A}$. 

We denote the set of all nonnegative integers by $\N$, and the set of all positive integers by $\Z^+$. Given $n\in\N$, let $\genset$ denote an $n$-element set. We shall refer to any element of $\genset$ as a \emph{letter}. For $t\in\N$, by a \emph{word of length} $t$ on $\genset$ we mean a sequence of the form
\begin{equation}
X_1X_2\cdots X_t,\label{assocform}
\end{equation}
where $X_i\in\genset$ for $1\leq i\leq t$. Given a word $W$ on $\genset$, denote the length of $W$ by $|W|$. The word of length $0$ is $\algI$. Let $\monoid$ denote the set of all words on $\genset$. Given words $X_1X_2\cdots X_s$ and $Y_1Y_2\cdots Y_t$ on $\genset$, their \emph{concatenation product} is $$X_1X_2\cdots X_sY_1Y_2\cdots Y_t.$$

We now recall the free $\F$-algebra $\F\monoid$. The $\F$-vector space $\F\monoid$ has basis $\monoid$. Multiplication in the $\F$-algebra $\F\monoid$ is the concatenation product.  

Given $n\in\N$, the subspace of $\F\monoid$ spanned by all the words of length $n$ is the $n$\emph{-homogenous component} of $\F\monoid$. Observe that $\F\monoid$ is the direct sum of all the $n$-homogenous components for all $n$. If $f$ is an element of the $m$-homogenous component and $g$ is an element of the $n$-homogenous component, then $fg$ is an element of the $(m+n)$-homogenous component. It follows that the set of all $n$-homogenous components of $\F\monoid$ for all $n$ is a grading of $\F\monoid$.

The following notation will be useful. Let $W=X_1X_2\cdots X_t$ denote a word on $\genset$. We define $W\rev$ to be the word  $X_tX_{t-1}\cdots X_1$.
Let $\theta$ denote the $\F$-linear map
\begin{eqnarray}
\theta:  \F\monoid & \rightarrow & \F\monoid,\nonumber\\
 W & \mapsto & (-1)^{|W|}W\rev,\label{thetadef}
\end{eqnarray}
for any word $W$. By \cite[p.~19]{Reut}, the map $\theta$ is the unique anti-automorphism of the $\F$-algebra $\F\monoid$ that sends $X$ to $-X$ for any letter $X$. 

For the rest of the section, we fix an ordering $<$ on the elements of $\genset$. The ordering $<$ can be extended into the so-called \emph{lexicographic ordering} on $\monoid$ as follows. Let $W=X_1X_2\cdots X_{|W|}$ and $V=Y_1Y_2\cdots Y_{|V|}$ denote nonempty words, where $X_i$, $Y_j$ are letters for all $i,j$. We write $W<V$ if for some positive integer $t$, that does not exceed the minimum of $\left\{ |W|, |V|\right\}$, we have $X_t<Y_t$ and $X_i=Y_i$ for all $i<t$. The term ``lexicographic" is suggestive of how the previously described process compares words as to how they are to be ``ordered in a dictionary." However, the logical rules described have no sensible conclusion when $V=WR$ for some nonempty word $R$. There are several ways of extending the lexicographic ordering to cover the said case. We note that described in \cite[p. 34]{Ufn}. Denote two elements of $\monoid$ by $W=X_1X_2\cdots X_{|W|}$ and $V=Y_1Y_2\cdots Y_{|V|}$ such that $X_i,Y_i\in\genset$ for any $i$. Then in the free algebra $\F\monoid$, we have $W=V$ if and only if $|W|=|V|$ and $X_i=Y_i$ for any $i$. We define $\sim$ to be the relation on $\monoid$ given by $W\sim V$ if and only if $WV=VW$. An immediate consequence is that $W\sim V$ if and only if $W$ and $V$ are powers of the same word, and this fact can be used to show that $\sim$ is an equivalence relation. We define the relation $\rtri$ on $\monoid$ by $V\rtri W$ if and only if $VW > WV$. Under this new relation on $\monoid$, the words $W,V$ do not belong to the same equivalence class under $\sim$ if and only if either $V\rtri W$ or $W\rtri V$. Thus, $\rtri$ is a partial ordering on $\monoid$. Also, the relation $\rtri$ may be interpreted as a well ordering of the equivalence classes of $\sim$.

Write a nonempty word $W$ as $W=X_1X_2\cdots X_{|W|}$ where $X_i$ is a letter for any $i$. Consider positive integers $s,t$ such that $s\leq t\leq |W|$. We call $X_sX_{s+1}\cdots X_t$ a \emph{subword} of $W$. The word $G=X_1X_2\cdots X_t$ is called a \emph{beginning} of $W$, and $H=X_{t}X_{t+1}\cdots X_{|W|}$ an \emph{ending} of $W$. If $t<|W|$, then $G$ and $H$ are a \emph{proper beginning} and a \emph{proper ending} of $W$, respectively. 

\begin{definition} A nonempty word $W$ is \emph{regular} (or is an \emph{associative regular word}) if for any proper beginning $G$ and proper ending $H$ of $W$ such that $W=GH$, we have $G\rtri H$.
\end{definition}

\begin{proposition}[{\cite[p. 35]{Ufn}}]\label{factorProp} Let $f$ be a regular word. Then $f$ is maximal among all its cyclic permutations with respect to the ordering $\rtri$. Furthermore, if $h$ is the length-maximal nonempty proper ending of $f$ that is also regular, and if $g$ is the word such that $f=gh$, then $g$ is regular. As described, any regular word of length at least $2$ can be expressed uniquely as a product of two regular proper subwords.
\end{proposition}

We now recall the notion of a \emph{Lie monomial} on $\genset$. The set of all {Lie monomials} on $\genset$ is the minimal subset of $\F\monoid$ that contains $\genset$ and is closed under the Lie bracket. Observe that $0$ is a Lie monomial. Let $U$ be a Lie monomial. Then $U$ is an element of some $n$-homogenous component of $\F\monoid$. Given a Lie monomial $U\neq 0$, we find that $U$ belongs to exactly one $n$-homogenous component. Let $\freeLie$ denote the Lie subalgebra of the Lie algebra $\F\monoid$ generated by $\genset$. Following \cite[Theorem~0.5]{Reut}, we call $\freeLie$ the \emph{free Lie algebra on} $\genset$. An element of $\freeLie$ is called a \emph{Lie polynomial on} $\genset$. The following gives us a useful necessary condition for Lie polynomials on $\genset$.

\begin{proposition}[\protect{\cite[Lemma 1.7]{Reut}}]\label{ReutProp} For $f\in\freeLie$, we have $\theta(f)=-f$.
\end{proposition}
We now recall a basis for $\freeLie$ that consists of Lie monomials and is related to regular words.

\begin{definition}\label{nonassocDef} Let $f$ be a regular word. If $|f|=1$, then we define $\lnon f\rnon :=f$. If $|f|\geq 2$, let $f=gh$ as described in Proposition~\ref{factorProp}. We define $\lnon f\rnon := \lbrack\lnon g\rnon,\lnon h\rnon\rbrack$. We call the Lie monomial $\lnon f\rnon$ a \emph{nonassociative regular word} on $\genset$. We also say that $\lnon f\rnon$ is the nonassociative regular word for the (associative) regular word $f$.
\end{definition}

\begin{proposition}[{\cite[p. 36]{Ufn}}] The nonassociative regular words on $\genset$ form a basis for $\freeLie$.
\end{proposition}

Throughout, by an \emph{ideal} of an $\F$-algebra $\mathfrak{A}$ we mean a two-sided ideal of $\mathfrak{A}$. By a \emph{Lie ideal} of a Lie algebra $\mathfrak{L}$ we mean an ideal of $\mathfrak{L}$ under the Lie algebra structure. We now recall the notion of algebras having generators and relations (i.e., having a presentation). Denote the elements of $\genset$ by $G_1,G_2,\ldots, G_n$. Let $f_1,f_2,\ldots ,f_m\in\F\monoid$ and let $I$ be the ideal of $\F\monoid$ generated by $f_1,f_2,\ldots ,f_m$. We define $\F\monoid/I$ as the $\F$-algebra with generators $G_1,G_2,\ldots ,G_n$ and relations $f_1=0,f_2=0,\ldots ,f_m=0$. The Lie subalgebra of $\F\monoid/I$ generated by $\genset$ is $\freeLie/\left(I\cap\freeLie\right)$.

Let $g_1,g_2,\ldots ,g_m\in\freeLie$ and let $J$ be the Lie ideal of $\freeLie$ generated by $g_1,g_2,\ldots ,g_m$. We define $\freeLie/J$ as the Lie algebra with generators $G_1,G_2,\ldots ,G_n$ and relations $g_1=0,g_2=0,\ldots ,g_m=0$. 

Suppose $\mathfrak{L}$ is a Lie algebra (over $\F$) generated by $\genset$. Then there exists an ideal $\mathcal{K}$ of $\freeLie$ such that $\mathfrak{L}=\freeLie/\mathcal{K}$. Let $\phi:\freeLie\rightarrow\freeLie/\mathcal{K}$ be the canonical Lie algebra homomorphism. Then $\mathfrak{L}$ has a spanning set consisting of all vectors of the form $\phi(U)$ where $U$ is a nonassociative regular word. We call this spanning set the nonassociative regular words in $\mathfrak{L}$. Given a Lie algebra $\mathfrak{L}$ and a fixed $x\in\mathfrak{L}$, recall the adjoint linear map 
\begin{eqnarray}
\ad x:  \mathfrak{L} & \rightarrow & \mathfrak{L},\nonumber\\
 y & \mapsto & \lbrack x,y\rbrack.\nonumber
\end{eqnarray}
Recall that given $x,y,z\in\mathfrak{L}$, the Jacobi identity $\lbrack x,\lbrack y,z\rbrack\rbrack + \lbrack z,\lbrack x,y\rbrack\rbrack + \lbrack y,\lbrack z,x\rbrack\rbrack =0 $ can be rewritten in terms of adjoint maps as
\begin{eqnarray}
\left(\ad \lbrack x,y\rbrack\right)(z) = \left(\left(\ad x\right)\circ\left(\ad y\right)\right)(z) - \left(\left(\ad y\right)\circ\left(\ad x\right)\right)(z).
\end{eqnarray}

Given subsets $\mathcal{K}_1$ and $\mathcal{K}_2$ of a Lie algebra $\mathfrak{L}$, we define $\lbrack \mathcal{K}_1,\mathcal{K}_2\rbrack$ as the span of all vectors of the form $\lbrack x,y\rbrack$ for some $x\in\mathcal{K}_1$ and some $y\in \mathcal{K}_2$. The \emph{derived algebra} of $\mathfrak{L}$ is the Lie algebra $\mathfrak{L}^{(1)} := \lbrack \mathfrak{L}, \mathfrak{L}\rbrack$. Set $\mathfrak{L}^{(0)}:=\mathfrak{L}$, and we inductively define $\mathfrak{L}^{(p)} := \lbrack \mathfrak{L}^{(p-1)}, \mathfrak{L}\rbrack$ for all $p\in\Z^+$. We say that $\mathfrak{L}$ is \emph{$p$-step nilpotent} whenever $p$ is the smallest integer such that $\mathfrak{L}^{(p)}$ is the zero Lie algebra. Observe that a Lie algebra is one-step nilpotent if and only if the Lie bracket of any two basis vectors is zero.

From this point onward, suppose we only have two generators $A,B$ for $\F\monoid$. Any word $W$ on two letters $A,B$ can be written as a product of powers of $A,B$. That is,
\begin{equation}\label{powerAB}
W = B^{m_1}A^{n_1}B^{m_2}A^{n_2}\cdots B^{m_k}A^{n_k},
\end{equation}
for some $k,m_1,n_1,m_2,n_2,\ldots,m_k,n_k\in\N$. Without loss of generality, we further assume that if $k\geq 2$, then $m_i$ and $n_j$ are nonzero for any $i\in\{2,3,\ldots,k\}$ and any $j\in\{1,2,\ldots,k-1\}$. By doing so, any such word $W$, if nonempty, is paired uniquely with the integer $k$ in \eqref{powerAB}. If $W$ is the empty word $\algI$, we view $k=0$. Thus, we define the function $\BActr :\monoid\rightarrow\Z$ given by the rule $W\mapsto k$ whenever $W$ is written as \eqref{powerAB}, with the exponents as previously described. Throughout, we use the ordering $A<B$ to construct the regular words on $A,B$. From these facts, we have an immediate consequence that if the word $W$ on $A,B$ is regular, then both $m_1$ and $n_{\BActr(W)}$ are nonzero. We note that given a regular word $W$, the number $\BActr(W)$ is precisely the number of occurences of the word $BA$ as a subword of $W$, which is usually denoted in the literature as $\deg_{BA}(W)$. i.e., $\BActr(W)=\deg_{BA}(W)$. This notation is from \cite[p. 27]{Ufn}. The following immediate consequence of Definition~\ref{nonassocDef} on $\genset = \{A,B\}$ shall be useful.

\begin{lemma}\label{bracketingLem} Let $m,n\in\N$. The following hold in the free Lie algebra over $\F$ generated by $A,B$.
\begin{eqnarray}
\lbrack\lnon BA^n\rnon, A\rbrack & = & \lnon BA^{n+1}\rnon,\label{Nprop1}\\
\lbrack B,\lnon B^mA^n\rnon\rbrack & = & \lnon B^{m+1}A^n\rnon,\quad (m,n\geq 1),\label{Nprop2}
\end{eqnarray}
\end{lemma}

Denote by $I_q$ the ideal of $\F\monoid$ generated by $AB-qBA-\algI$. Thus, the $q$-deformed Heisenberg algebra is $\Heisen=\F\monoid/I_q$. By contraposition of Proposition~\ref{ReutProp}, it is routine to show that if $q\neq 1$ or $\charF\neq 2$, then $AB-qBA-\algI\notin\freeLie$. This implies that the $q$-deformed commutation relation $AB-qBA=\algI$, under the said cases, is not expressible in terms of Lie algebra operations only. Equivalently, $\freeLie$ has, by definition, no Lie ideal generated by $AB-qBA-\algI$. Our goal is to study the Lie algebra $\freeLie/\left(I_q\cap\freeLie\right)$. We shall refer to the elements of $\freeLie/\left(I_q\cap\freeLie\right)$ as the \emph{Lie polynomials in $\Heisen$}.

\section{Some properties of $\Heisen$}\label{PropSec}
We recall the following expressions from the so-called \emph{$q$-special combinatorics} \cite[Section C.1]{Hel}. For a  given $n\in\Z$ and $z\in\F$,
\begin{eqnarray}
\{n\}_z & := & \sum_{l=0}^{n-1} z^l,\label{qnatural}\\
\{n\}_z! & := & \prod_{l=1}^n\{ l\}_z,\label{qfactorial}\\
{{n}\choose{i}}_z & := & \frac{\{n\}_z!}{\{i\}_z!\{n-i\}_z!}.
\end{eqnarray}
If $n\leq 0$, then we interpret \eqref{qnatural} as the empty sum $0$ and \eqref{qfactorial} as the empty product $1$. One of the most important properties of $\Heisen$ is that the vectors $B^mA^n$ ($m,n\geq 0$) form a basis for $\Heisen$ \cite[Theorem 3.1]{Hel}. The proof is by the Diamond Lemma for Ring Theory \cite{Berg}. Products of these basis elements can be written as linear combinations of the same using  the following ``reordering formulae" from \cite[Proposition 4.1]{Hel05}. If $n\in\Z^+$, then
\begin{eqnarray} 
AB^n & = & q^nB^nA + \{n\}_q B^{n-1},\label{startB}\\
A^nB & = & q^nBA^n + \{n\}_q A^{n-1},\label{startA}\\
BA^n & = & q^{-n} A^nB - q^{-1}\{ n\}_{q^{-1}}A^{n-1},\quad (q\neq 0),\\
B^nA & = & q^{-n} AB^n - q^{-1}\{ n\}_{q^{-1}}B^{n-1},\quad (q\neq 0).\label{reorBnA}
\end{eqnarray}
For any subspaces $H,K$ of $\Heisen$, define $HK:=\Span\{hk\  |\    h\in H,\  k\in K\}$. Recall that a \emph{$\Z$-grading} of $\Heisen$ is a sequence $\{\mathcal{A}_n\}_{n\in\Z}$ of subspaces such that $\Heisen=\sum_{n\in\Z} \mathcal{A}_n$ is a direct sum and that $\mathcal{A}_m\mathcal{A}_n\subseteq \mathcal{A}_{m+n}$ \cite[p. 202]{Cart}. For any $n\in\Z$, define $\filtrH_n :=\Span\{B^kA^l\ |\  k,l\in\N,\  k-l=n\}$. Then $\{\filtrH_n\}_{n\in\Z}$ is a $\Z$-grading of $\Heisen$ \cite[Proposition 2.4]{Hel05}. 
\begin{proposition}[{\cite[Corollary 4.5]{Hel05}}]\label{gradProp} If $q\notin\{0,1\}$, then the vectors $$\LieAB^k, \  B^l\LieAB^k,\  \LieAB^kA^l,\quad k\in\N,\  l\in\Z^+, $$
form a basis for $\Heisen$. Furthermore, the vectors $\LieAB^k$ ($k\in\N$) form a basis for the subalgebra $\filtrH_0$ of $\Heisen$. Given $l\in\Z^+$, the vectors $B^l\LieAB^k$ ($k\in\N$) form a basis for the subspace $\filtrH_l$ of $\Heisen$, while the vectors $\LieAB^kA^l$ ($k\in\N$) form a basis for the subspace $\filtrH_{-l}$.
\end{proposition}
Some useful relations involving the basis vectors in the above proposition are in the following.
\begin{proposition}[{\cite[pp. 105, 111]{Hel05}}] Let $n\in\N$. If $P$ is any polynomial in two non-commuting variables and $q\neq 0$, then 
\begin{eqnarray}
P(B,A)\LieAB^n = \LieAB^n P(q^{-n}B,q^nA).\label{shiftAB}
\end{eqnarray}  
Furthermore, if $q\notin\{0,1\}$, then
\begin{eqnarray}
B^nA^n = q^{-{{n}\choose{2}}}(q-1)^{-n}\sum_{i=0}^n(-1)^{n-i}q^{{n-i}\choose{2}}{{n}\choose{i}}_q\lbrack A,B\rbrack^i.\label{BnAntoLie}
\end{eqnarray}
\end{proposition}
In the computations that will follow, we would also need a relation for expressing the similar product $A^nB^n$ in terms of the basis vectors given in Proposition~\ref{gradProp}. We exhibit such in the following proposition, with proof similar to how \eqref{BnAntoLie} was proven in \cite[pp. 111-112]{Hel05}.

\begin{proposition} If $q\notin\{0,1\}$, then for any $n\in\N$,
\begin{eqnarray}
A^nB^n = (q-1)^{-n}\sum_{i=0}^n(-1)^{n-i}q^{{i+1}\choose{2}}{{n}\choose{i}}_q\lbrack A,B\rbrack^i.\label{AnBntoLie}
\end{eqnarray}
\end{proposition}
\begin{proof} We first claim that for any $n\in\N$,
\begin{eqnarray}
q^{-{{n}\choose{2}}}A^nB^n = \prod_{i=0}^{n-1}\left(AB+q^{-1}\{i\}_{q^{-1}}\algI\right).\label{AnBneq}
\end{eqnarray}
If $n=0$, then the left side of \eqref{AnBneq} reduces to the identity $\algI$, while the right side is the empty product, which is also $\algI$ in the algebra $\Heisen$. Thus, \eqref{AnBneq} holds for $n=0$. Suppose \eqref{AnBneq} holds for $n$. We have the following computations.
\begin{eqnarray}
\prod_{i=0}^{n}\left(AB+q^{-1}\{i\}_{q^{-1}}\algI\right)   & = & q^{-{{n}\choose{2}}} A^nB^n(AB+q^{-1}\{n\}_{q^{-1}}\algI),\nonumber\\
& = & q^{-{{n}\choose{2}}} A^n(B^nAB+q^{-1}\{n\}_{q^{-1}}B^n).\label{reorBnAB}
\end{eqnarray} 
We use the reordering formula \eqref{reorBnA} on the expression $B^nAB$ in the right side of \eqref{reorBnAB}. The resulting right side is $q^{-{{n+1}\choose{2}}}A^{n+1}B^{n+1}$. By induction, \eqref{AnBneq} holds for any $n$. By the defining relation of $\Heisen$, we have the equation $AB=(q-1)^{-1}(q\LieAB-\algI)$, which we use to rewrite $AB+q^{-1}\{i\}_{q^{-1}}\algI$ in \eqref{AnBneq} as a linear combination of $\LieAB$ and $\algI$. This results to 
\begin{eqnarray}
q^{-{{n}\choose{2}}}A^nB^n & = &  \prod_{i=0}^{n-1}(q-1)^{-1}\left(q\LieAB - q^{-i}\algI\right),\nonumber\\
& = & (q-1)^{-n}\prod_{i=0}^{n-1}\left(q\LieAB - q^{-i}\algI\right). \label{AnBneq2}
\end{eqnarray}
The product $\prod_{i=0}^{n-1}\left(q\LieAB - q^{-i}\algI\right)$ in \eqref{AnBneq2} is the \emph{Gauss polynomial} $G_n(q\LieAB;q^{-1})$. That is,
\begin{eqnarray}
A^nB^n = q^{{{n}\choose{2}}}(q-1)^{-n} G_n(q\LieAB;q^{-1}).\label{subGauss}
\end{eqnarray}
We recall that an important property of the Gauss polynomial $G_n(x;z)$ on the indeterminate $x$ and constant parameter $z$ is the following relation \cite[Theorem C.17]{Hel}.
\begin{eqnarray}
G_n(x;z) = \sum_{i=0}^n(-1)^{n-i}z^{{{n-i}\choose{2}}}{{{n}\choose{i}}_z}x^i.\label{Gaussrel}
\end{eqnarray}
Setting $x=q\LieAB$ and $z=q^{-1}$ in \eqref{Gaussrel} and substituting in \eqref{subGauss}, we obtain
\begin{eqnarray}
A^nB^n = q^{{n}\choose{2}}(q-1)^{-n} \sum_{i=0}^n(-1)^{n-i}q^{-{{n-i}\choose{2}}}{{{n}\choose{i}}_{q^{-1}}}q^i\LieAB^i.\label{AnBneq3}
\end{eqnarray}
From \cite[Theorem C.10]{Hel}, we have the identity ${{n}\choose{i}}_{q^{-1}}=q^{-i(n-i)}{{n}\choose{i}}_q$, which we use to replace ${{n}\choose{i}}_{q^{-1}}$ in \eqref{AnBneq3}. The resulting coefficient of $\LieAB^i$ in the summation is $$(-1)^{n-i}q^{i-{{n-i}\choose{2}}-i(n-i)}{{n}\choose{i}}_q, $$
where the exponent of $q$, by some routine computations, can be simplified into ${{i+1}\choose{2}}-{{n}\choose{2}}$. By these observations, we get \eqref{AnBntoLie} from \eqref{AnBneq3} .\qed
\end{proof}

\begin{lemma}\label{adadLem} Let $m,n\in\N$. The following hold in $\Heisen$.
\begin{eqnarray}
\lbrack\lnon BA\rnon,B^mA^n\rbrack & = & (q-1)(q^n-q^m)B^{m+1}A^{n+1}+(q^n-q^m)B^mA^n,\label{fadBA}\\
\lbrack B,B^mA^n\rbrack & = &  (1-q^n)B^{m+1}A^n-\{n\}_qB^mA^{n-1},\label{fadB}
\end{eqnarray}
\end{lemma}
\begin{proof} We first prove \eqref{fadBA}. We consider the case in which both $m,n$ are nonzero. The proofs are similar for the cases in which one of $m,n$ is zero. Observe that
\begin{eqnarray}
\lbrack\lnon BA\rnon,B^mA^n\rbrack & = & \lbrack BA-AB,B^mA^n\rbrack,\nonumber\\
& = & B(AB^m)A^n - B^m(A^nB)A - (AB^{m+1})A^n + B^m(A^{n+1}B).\label{parenth}
\end{eqnarray}
We replace all instances of a product of the form $AB^j$ or $A^iB$ in \eqref{parenth} by the appropriate formula from \eqref{startB} or \eqref{startA}. The result is a linear combination of only the words $B^{m+1}A^{n+1}$ and $B^mA^n$, where the scalar coefficient of $B^{m+1}A^{n+1}$ is $q^m-q^{m+1}-q^n+q^{n+1}$ while that of $B^mA^n$ is $\{m\}_q-\{m+1\}_q-\{n\}_q+\{n+1\}_q$. By some routine calculations, these can simplified into $(q-1)(q^n-q^m)$ and $q^n-q^m$, respectively. Thus, \eqref{fadBA} holds. The proof for \eqref{fadB} is similar.\qed
\end{proof}

Our next goal is to derive some relations for expressing the regular nonassociatve words $\lnon BA^n\rnon$ and $\lnon B^nA\rnon$ in terms of the basis $\{B^iA^j\  |\  i,j\in\N\}$ of $\Heisen$.

\begin{proposition} For any positive integer $n$, the following relations hold in $\Heisen$.
\begin{eqnarray}
\lnon BA^n\rnon & = & (1-q)^nBA^n - (1-q)^{n-1}A^{n-1},\label{fBAn}\\
\lnon B^nA\rnon & = & (1-q)^nB^nA - (1-q)^{n-1}B^{n-1}.\label{fBnA}
\end{eqnarray}
\end{proposition}
\begin{proof} We use induction on $n\geq 1$. Using the relation $AB-qBA=\algI$, we have $\lnon BA\rnon = BA - AB = (1-q)BA - \algI$, which proves \eqref{fBAn} and \eqref{fBnA} for $n=1$. Suppose that \eqref{fBAn} holds for some $n$. If we apply the map $-\ad A$ on both sides of the equation, by \eqref{Nprop1}, the left side becomes $\lnon BA^{n+1}\rnon$ while the right side becomes $(1-q)^n\left( BA^{n+1}-(AB^n)A\right)$, in which we replace the expression $AB^n$ in the second term using \eqref{startB}. After combining some terms, the resulting equation serves as proof that \eqref{fBAn} holds for $n+1$. By a similar argument, it can be shown that \eqref{fBnA} holds for $n+1$ if \eqref{fBnA} is assumed to hold for $n$. By induction, we get the desired result.\qed
\end{proof}

\section{The Lie algebra $\freeHLie$}\label{zeroLieSec}

In this section, we show a basis and some corresponding commutation relations for the Lie algebra $\freeHLie$. First, we mention some routine consequences of $q=0$. The reordering formulae are now
\begin{eqnarray}
AB^n & = & B^{n-1},\\
A^nB & = & A^{n-1}.
\end{eqnarray}
For nonzero integers $m,n$ the above can be generalized into 
\begin{eqnarray}
A^nB^m & = &\left\{
\begin{array}{ll} A^{n-m}, & n\geq m,\\
B^{m-n}, & n<m.\end{array}\right.
\label{ferAnBm}
\end{eqnarray}
From Lemma~\ref{adadLem}, for positive integers $m,n$, the images of the basis vectors $B^mA^n$ of $\freeH$ under the linear maps $\ad \lnon BA\rnon$, $\ad B$, and $\ad B^m$ are given by the following.
\begin{eqnarray}
\lbrack\lnon BA\rnon, B^mA^n\rbrack & = & 0,\label{adBA0}\\ 
\lbrack\lnon BA\rnon, B^m\rbrack & = & -B^{m+1}A+B^m = -\lnon B^{m+1}A\rnon,\label{adBA1}\\
\lbrack\lnon BA\rnon, A^n\rbrack & = & BA^{n+1}-A^n=\lnon BA^{n+1}\rnon,\label{adBA2}\\
\lbrack B, B^mA^n\rbrack & = & B^{m+1}A^n-B^mA^{n-1},\label{adB0}\\ 
\lbrack B,A^n\rbrack & = & BA^n-A^{n-1} = \lnon BA^n\rnon,\\
\lbrack B^m,A\rbrack & = & B^mA-B^{m-1} = \lnon B^mA\rnon.
\end{eqnarray}

\begin{proposition}\label{freeBmAnProp} For any $m,n\in\Z^+$,
\begin{eqnarray}
\lnon B^mA^n\rnon = \sum_{i=0}^h (-1)^i{{h}\choose{i}} B^{m-i}A^{n-i},\label{BmAnEQ}
\end{eqnarray}
where $h$ is the minimum of $\{ m,n\}$.
\end{proposition}
\begin{proof} For the case $n\geq m$, use \eqref{fBAn}, \eqref{adB0}, and induction on $m$. For the case $n<m$, we shall use induction on $m-n$. First, use the previous case to obtain the formula for $\lnon B^nA^n\rnon$. Apply $\ad B$ on both sides of the equation. This leads to the proof for $m-n=1$. The rest follows by induction on $m-n$. \qed
\end{proof}

\begin{corollary}\label{adBAinLieCor} The following relations hold in $\freeH$ for any $m,n\in\Z^+$.
\begin{eqnarray}
\lbrack\lnon BA\rnon, \lnon B^mA^n\rnon\rbrack & = &\left\{
\begin{array}{ll} -\lnon B^{m-n+1}A\rnon, & n < m,\\
0, & n = m,\\
\lnon BA^{n-m+1}\rnon, & n > m.\end{array}\right.
\label{fer2AnBm}
\end{eqnarray}
\end{corollary}
\begin{proof} In view of \eqref{BmAnEQ}, the only basis vector of the form $B^xA^y$ that appears in the linear combination for $\lnon B^mA^n\rnon$ with $xy = 0$ is for the index $i=h=\min\{m,n\}$, which is either $B^{m-n}$, $\algI$, or $A^{n-m}$. Then in view of \eqref{adBA0}, if we apply $\ad\lnon BA\rnon$ on both sides of \eqref{BmAnEQ}, all terms of the form $\lbrack\lnon BA\rnon, B^xA^y\rbrack$ with $xy\neq 0$ will vanish. We are left with $\lbrack\lnon BA\rnon, B^{m-n}\rbrack$ if $n<m$, with $0$ if $n=m$, or with $\lbrack\lnon BA\rnon, A^{n-m}\rbrack$ if $n>m$. To get the desired result, use \eqref{adBA1} and \eqref{adBA2} on the first and third of the said cases, respectively.\qed
\end{proof}

\begin{corollary}\label{freeindepGradCor} Let $l\in\N$. The following vectors form a basis for the $\Z$-grading subspace $\filtrH_l$ of $\freeH$.
\begin{eqnarray}
B^l, \  \lnon B^mA^n\rnon,\quad m,n\in\Z^+,\  m-n=l,\label{freeplusBasis}
\end{eqnarray}
while the following vectors form a basis for $\filtrH_{-l}$, 
\begin{eqnarray}
A^l, \  \lnon B^mA^n\rnon,\quad m,n\in\Z^+,\  m-n=-l.\label{freeminusBasis}
\end{eqnarray}
\end{corollary}
\begin{proof} We show that the vectors in \eqref{freeplusBasis} form a basis for $\filtrH_l$. Since $m-n=l$, we rewrite   $B^mA^n$ as $ B^{n+l}A^n$. We claim that for any $n$, there exists a bijective linear transformation of the span of 
\begin{eqnarray}
B^l, \  B^{h+l}A^h,\quad h\leq n, \label{inductGrad1}
\end{eqnarray} 
onto the span of 
\begin{eqnarray}
B^l, \  \lnon B^{h+l}A^h\rnon,\quad h\leq n.\label{inductGrad2}
\end{eqnarray}
We prove this claim by induction on $n$. Consider $n=1$. By \eqref{BmAnEQ}, 
\begin{eqnarray}
\lnon B^{1+l}A\rnon = B^{1+l}A - B^l.
\end{eqnarray}
Because of the above equation, we immediately find that there is a transition matrix, with determinant $1$, of the span of $B^l, B^{1+l}A$ onto the span of $B^l,\lnon B^{1+l}A\rnon$. Thus, the claim holds for $n=1$. Suppose the claim holds for some $n$. Our goal is to show that there is a bijective linear transformation of the span of \begin{eqnarray}
B^l, \  B^{h+l}A^h,\quad h\leq n+1,\label{inductGrad3}
\end{eqnarray}
onto the span of 
\begin{eqnarray}
B^l, \  \lnon B^{h+l}A^h\rnon.\quad h\leq n+1.\label{inductGrad4}
\end{eqnarray}
Observe that the only vector that is in \eqref{inductGrad3} that is not in \eqref{inductGrad1} is $ B^{1+n+l}A^{n+1}$, while the only vector  in \eqref{inductGrad4} that is not in \eqref{inductGrad2} is $\lnon B^{1+n+l}A^{n+1}\rnon$. Also, by the inductive hypothesis there exists an upper triangular transition matrix $T$ from the span of \eqref{inductGrad1} onto the span of \eqref{inductGrad2} with determinant $c\neq 0$. These observations imply that we need only one equation expressing $\lnon B^{1+n+l}A^{n+1}\rnon$ as a linear combination of \eqref{inductGrad3}. We get this from \eqref{BmAnEQ}, and this may be used to construct an upper triangular transition matrix $T'$ also with determinant $c\neq 0$ from the span of \eqref{inductGrad3} onto the span of \eqref{inductGrad4}. By induction, this proves the claim, and since it holds for all $n$, we find that there is a bijective linear transformation from the span of (the basis) $\{B^mA^n\  |\  m-n=l\}$ of $\filtrH_l$ onto the subspace of $\filtrH_l$ spanned by \eqref{freeplusBasis}. Therefore, \eqref{freeplusBasis} is a basis for $\filtrH_l$. The proof for $\filtrH_{-l}$ is similar.\qed
\end{proof}

\begin{corollary}\label{freeindepCor} The following vectors form a basis for $\freeH$.
\begin{eqnarray}
\algI, A^n, B^m, \  \lnon B^mA^n\rnon,\quad m,n\geq 1.\label{freeLieBasis}
\end{eqnarray}
\end{corollary}
\begin{proof} Use Corollary~\ref{freeindepGradCor} and the fact that $\{\filtrH_n\}_{n\in\Z}$ is a $\Z$-grading of $\freeH$.\qed
\end{proof}

\begin{corollary}\label{notAnBmCor} The following relations hold in $\freeH$ for any $m,n\in\Z^+$ with $m,n\geq 4$.
\begin{eqnarray}
B^2A^n - A^{n-2} & = & \lnon B^2A^n\rnon + 2\lnon BA^{n-1}\rnon,\label{notAnBmeq1}\\
B^mA^2-B^{m-2} & = & \lnon B^mA^2\rnon + 2\lnon B^{m-1}A\rnon.\label{notAnBmeq2}
\end{eqnarray}
\end{corollary}
\begin{proof} Use \eqref{fBAn}, \eqref{fBnA}, and \eqref{BmAnEQ}.\qed
\end{proof}

\begin{corollary}\label{zeroId0Cor} The following vectors form a basis for $\freeH$.
\begin{eqnarray}
\algI, \  A, B,\  B^2A^j, B^iA^2, \  \lnon B^mA^n\rnon,\quad m,n\geq 1,\  i,j\geq 4.\label{extendbasis}
\end{eqnarray}
\end{corollary}
\begin{proof} An immediate consequence of the equations \eqref{notAnBmeq1} and \eqref{notAnBmeq2} is that for $m,n\geq 2$, if $A^{n-2}$ is replaced by $B^2A^n$ and $B^{m-2}$ by $B^mA^2$, we get a bijective linear map of the span of the basis $\{B^iA^j\  |\  i,j\in\N\}$ of $\freeH$ onto the span of \eqref{extendbasis}, and so $\freeH$ has a basis consisting of the vectors in \eqref{extendbasis}.\qed
\end{proof}

\begin{theorem}\label{freeThm} The Lie algebra $\freeHLie$ has a basis consisting of the vectors 
\begin{eqnarray}
A,B,\  \lnon B^mA^n\rnon,\quad (m,n\in\Z^+).\label{freeBasis}
\end{eqnarray}
\end{theorem}
\begin{proof} We first claim that the following Lie monomials are elements of the span of \eqref{freeBasis} for all $m,n,i,j\in\Z^+$.
\begin{eqnarray}
\lbrack B,\lnon B^mA^n\rnon\rbrack \label{BBmAnREL}\\
\lbrack \lnon B^mA^n\rnon, A\rbrack \label{BmAnAREL}\\
\lbrack \lnon B^mA^n\rnon,\lnon B^iA^j\rnon\rbrack \label{zeroREL}.
\end{eqnarray}
The Lie monomial \eqref{BBmAnREL} is equal to $\lnon B^{m+1}A^n\rnon\in\freeHLie$ by the definition of a nonassociative regular word. We show \eqref{BmAnAREL} is an element of $\freeHLie$ by induction on $m$. The case $m=1$ follows immediately from the definition of $\lnon BA^n\rnon$. Suppose \eqref{BmAnAREL} is in $\freeHLie$ for some $m$. Then by the properties of the Lie bracket,
\begin{eqnarray}
\lbrack \lnon B^{m+1}A^n\rnon, A\rbrack & = & \lbrack \lbrack B, \lnon B^{m}A^n\rnon\rbrack, A\rbrack,\nonumber\\
& = & \lbrack  B, \lbrack\lnon B^{m}A^n\rnon, A\rbrack\rbrack - \lbrack\lnon B^{m}A^n\rnon, \lbrack B,A\rbrack\rbrack,\nonumber\\
& = & \lbrack  B, \lbrack\lnon B^{m}A^n\rnon, A\rbrack\rbrack + \lbrack \lnon BA\rnon,\lnon B^{m}A^n\rnon\rbrack.\label{zeroterm2}
\end{eqnarray}
By Corollary~\ref{adBAinLieCor} the second term in \eqref{zeroterm2} is in $\freeHLie$. As for the remaining term, the inductive hypothesis applies to $\lbrack\lnon B^{m}A^n\rnon, A\rbrack$, and so $\lbrack  B, \lbrack\lnon B^{m}A^n\rnon, A\rbrack\rbrack$ is in $\freeHLie$. By induction, we find that \eqref{BmAnAREL} is in $\freeHLie$. We now prove \eqref{zeroREL} is in $\freeHLie$. Let $f=\lnon B^mA^n\rnon$. First, we consider the case $i=1$, and use induction on $j$. For $j=1$, we simply use \eqref{adBA0} and \eqref{BmAnEQ}. Suppose \eqref{zeroREL} is in $\freeHLie$ for $i=1$ and some $j$. Then
\begin{eqnarray}
\lbrack f,\lnon BA^{j+1}\rnon\rbrack & = & \lbrack f,\lbrack\lnon BA^j\rnon , A\rbrack\rbrack\nonumber\\
& = & - \lbrack A,\lbrack f, \lnon BA^j\rnon\rbrack\rbrack  - \lbrack \lnon BA^j\rnon,\lbrack A,f \rbrack\rbrack,\nonumber\\
& = & - \lbrack A,\lbrack f, \lnon BA^j\rnon\rbrack\rbrack - \lbrack \lbrack f,A \rbrack, \lnon BA^j\rnon\rbrack,\nonumber\\
& = & - \lbrack A,\lbrack f, \lnon BA^j\rnon\rbrack\rbrack - \lbrack \lbrack \lnon B^mA^n\rnon,A \rbrack,\lnon BA^j\rnon\rbrack.\label{inductBAj}
\end{eqnarray}
We use \eqref{BmAnAREL} on the expression $\lbrack \lnon B^mA^n\rnon,A \rbrack$, which is in the second term in \eqref{inductBAj}. This gives us
\begin{eqnarray}
\lbrack f,\lnon BA^{j+1}\rnon\rbrack = - \lbrack A,\lbrack f, \lnon BA^j\rnon\rbrack\rbrack - \lbrack  \lnon B^mA^{n+1}\rnon,\lnon BA^j\rnon\rbrack,\label{inductBAj2}
\end{eqnarray}
wherein the inductive hypothesis applies to both $\lbrack f, \lnon BA^j\rnon\rbrack$ and $\lbrack  \lnon B^mA^{n+1}\rnon,\lnon BA^j\rnon\rbrack$. Then the right side of \eqref{inductBAj2} is in $\freeHLie$. By induction, \eqref{zeroREL} is in $\freeHLie$ for $i=1$ and all $j$.  We now proceed with induction on $i$. By the previous part of the proof, we are done with $i=1$. Suppose that for some $i$, \eqref{zeroREL} is in $\freeHLie$ for all $j$. We further have
\begin{eqnarray}
\lbrack f,\lnon B^{i+1}A^{j}\rnon\rbrack & = &\lbrack f,\lbrack B,\lnon B^iA^j\rnon\rbrack\rbrack\nonumber, \\
& = & - \lbrack \lnon B^iA^j\rnon,\lbrack f, B\rbrack\rbrack - \lbrack B,\lbrack \lnon B^iA^j\rnon,f\rbrack\rbrack\nonumber, \\
& = & - \lbrack \lbrack B,f\rbrack,\lnon B^iA^j\rnon\rbrack + \lbrack B,\lbrack f,\lnon B^iA^j\rnon\rbrack\rbrack,\nonumber\\
& = & - \lbrack \lbrack B,\lnon B^mA^n\rnon\rbrack,\lnon B^iA^j\rnon\rbrack + \lbrack B,\lbrack f,\lnon B^iA^j\rnon\rbrack\rbrack,\nonumber\\
& = & - \lbrack \lnon B^{m+1}A^n\rnon,\lnon B^iA^j\rnon\rbrack + \lbrack B,\lbrack f,\lnon B^iA^j\rnon\rbrack\rbrack,\label{inductBiAj}
\end{eqnarray}
where the inductive hypothesis applies to $\lbrack \lnon B^{m+1}A^n\rnon,\lnon B^iA^j\rnon\rbrack$ and $\lbrack f,\lnon B^iA^j\rnon\rbrack$. Then the right side of \eqref{inductBiAj} is in $\freeHLie$. By induction, \eqref{zeroREL} is in $\freeHLie$ for any $i,j$, and our claim is proven.
We now prove that the vectors in \eqref{freeBasis} form a basis for $\freeLie$. Linear independence follows from Corollary~\ref{freeindepCor}. Also, observe that by construction, all the vectors in \eqref{freeBasis} are in $\freeHLie$, and so to prove spanning, we show that for any regular word $W$ on $B>A$, the commutator $\lnon W\rnon$ is in the span of the vectors in \eqref{freeBasis}. This is because all the nonassociative regular words on $A<B$ form a spanning set for $\freeHLie$. Given a regular word $W$, we use induction on $\BActr(W)$. The case $\BActr(W)=1$ is trivial. If $\BActr(W)\geq 2$, then $W=UV$ for some regular words $U,V$ such that $\lnon W\rnon = \lbrack \lnon U\rnon, \lnon V\rnon\rbrack$ and that $\BActr(U),\  \BActr(V)<\BActr(W)$. By the inductive hypothesis, both $\lnon U\rnon$ and $\lnon V\rnon$ are in the span of \eqref{freeBasis}. Then in taking the Lie bracket $\lbrack \lnon U\rnon, \lnon V\rnon\rbrack$, by the claim, the result is still a linear combination of the vectors in \eqref{freeBasis}. Then $\lnon W\rnon$ is in the span of \eqref{freeBasis}. By induction, we get the desired result.\qed
\end{proof}

\begin{lemma}\label{zeroId1Lem} The following relations hold in $\freeH$ for any $i,j,k,l\geq 4$.
\begin{eqnarray}
\lbrack B,B^2A^k\rbrack & = & \lnon B^3A^k\rnon + 2\lnon B^2A^{k-1}\rnon-3\lnon BA^{k-2}\rnon,\label{idLemeq1}\\
\lbrack B^2A^k,A\rbrack & = & \lnon B^2A^{k+1}\rnon+\lnon BA^k\rnon,\label{idLemeq2}\\
\lbrack B^iA^2, A\rbrack & = & \lnon B^iA^3\rnon+2\lnon B^{i-1}A^2\rnon-3\lnon B^{i-2}A\rnon,\label{idLemeq3}\\
\lbrack B,B^lA^2\rbrack & = & \lnon B^{l+1}A^2\rnon+\lnon B^lA\rnon,\label{idLemeq4}\\
\lbrack B^iA^2,B^lA^2\rbrack & = & 0\quad =\quad\lbrack B^2A^j, B^2A^k\rbrack.\label{idLemeq5}
\end{eqnarray}
\end{lemma}
\begin{proof} We first prove \eqref{idLemeq1}. By \eqref{ferAnBm}, we have $\lbrack B,B^2A^k\rbrack = B^3A^k-B^2A^{k-1}$. Use \eqref{fBAn} and \eqref{BmAnEQ} to eliminate $B^3A^k$ and $B^2A^{k-1}$, and also $BA^{k-2}, A^{k-3}$ in succeeding steps so that only Lie monomials remain. The result is \eqref{idLemeq1}. The proofs for \eqref{idLemeq2}, \eqref{idLemeq3} \eqref{idLemeq4} are similar, while \eqref{idLemeq5} is a simple consequence of \eqref{ferAnBm}.\qed
\end{proof}

\begin{corollary}\label{zeroId2Cor} The following are elements of $\freeHLie$ for any $i,j,k,l\geq 4$.
\begin{eqnarray}
\lbrack \lnon B^{i}A^{j}\rnon, B^lA^2\rbrack,\label{idCoreq6}\\
\lbrack \lnon B^{i}A^{j}\rnon, B^2A^k\rbrack.\label{idCoreq7}
\end{eqnarray}
\end{corollary}
\begin{proof} We first prove \eqref{idCoreq6} by induction on $i,j$. If $i=1=j$, then by \eqref{adBA0}, the result is $0\in\freeHLie$. Assume that for $i=1$, the relation \eqref{idCoreq6} holds for some $j$. By the Jacobi identity,
\begin{eqnarray}
\lbrack \lnon BA^{j+1}\rnon, B^lA^2\rbrack &=& \lbrack\lbrack\lnon BA^j\rnon, B^lA^2\rbrack, A\rbrack - \lbrack \lnon BA^n\rnon,\lbrack B^lA^2,A\rbrack\rbrack,
\end{eqnarray}
in which the inductive hypothesis applies to $\lbrack\lnon BA^j\rnon, B^lA^2\rbrack$ in the first term, while in the second term, $\lbrack B^lA^2,A\rbrack\in\freeHLie$ by \eqref{idLemeq3}. Then $\lbrack \lnon BA^{j+1}\rnon, B^lA^2\rbrack\in\freeHLie$, and \eqref{idCoreq6} holds for all $j$ if $i=1$. We further assume that for some $i$, \eqref{idCoreq6} holds for any $j$. We have
\begin{eqnarray}
\lbrack \lnon B^{i+1}A^{j}\rnon, B^lA^2\rbrack &=& \lbrack B,\lbrack \lnon B^iA^j\rnon, B^lA^2\rbrack\rbrack - \lbrack \lnon B^iA^j\rnon,\lbrack B,B^lA^2\rbrack\rbrack.
\end{eqnarray}
Using the inductive hypothesis and \eqref{idLemeq4}, we find that $\lbrack \lnon B^{i+1}A^{j}\rnon, B^lA^2\rbrack\in\freeHLie$. By induction, we get the desired result. The proof for \eqref{idCoreq7} is similar.\qed
\end{proof}

\begin{theorem} The Lie subalgebra $\freeHLie$ of $\freeH$ is a Lie ideal of $\freeH$, and the quotient Lie algebra $\freeH/\freeHLie$ is infinite-dimensional and one-step nilpotent.
\end{theorem}
\begin{proof} In view of Corollary~\ref{zeroId0Cor} to Corollary~\ref{zeroId2Cor}, $\freeHLie$ is a Lie ideal of $\freeH$ and the quotient Lie algebra $\freeH/\freeHLie$ is spanned by $\algI, B^iA^2, B^2A^j$ $(i,j\geq 4)$. The said vectors are linearly independent in $\freeH/\freeHLie$. Otherwise, some linear combination of these vectors in $\freeH$ is an element of $\freeHLie$, contradicting Corollary~\ref{zeroId0Cor}. Then $\algI, B^iA^2, B^2A^j$ $(i,j\geq 4)$ form a basis for $\freeH/\freeHLie$, and hence, $\freeH/\freeHLie$ is inifinite-dimensional. We show $\freeH/\freeHLie$ is one-step nilpotent. Define
\begin{eqnarray}
f_1 & := & BA-\algI \quad =\quad \lnon BA\rnon,\\
f_r & := & B^rA^r - I = B^{r-1}\lnon BA\rnon A^{r-1} + f_{r-1},\quad r\geq 2.
\end{eqnarray}
Setting $q=0$ in \eqref{fBnA} gives us $B^{r-1}\lnon BA\rnon A^{r-1}=\lnon B^rA\rnon A^{r-1}$. But because $$\lbrack\lnon B^rA\rnon, A^{r-1}\rbrack = \lnon B^rA\rnon A^{r-1}- A^{r-1}\lnon B^rA\rnon, $$ we further have
\begin{eqnarray}
B^{r-1}\lnon BA\rnon A^{r-1} &=& \lbrack\lnon B^rA\rnon, A^{r-1}\rbrack +  A^{r-1}\lnon B^rA\rnon,\nonumber\\
&=& \lbrack\lnon B^rA\rnon, A^{r-1}\rbrack +  A^{r-1}B^rA-A^{r-1}B^{r-1},\nonumber\\
&=& \lbrack\lnon B^rA\rnon, A^{r-1}\rbrack +  BA-\algI,\nonumber\\
&=& \lbrack\lnon B^rA\rnon, A^{r-1}\rbrack +  \lnon BA\rnon.\label{nilpotenteq}
\end{eqnarray}
Since $\freeHLie$ is a Lie ideal, $\lbrack\lnon B^rA\rnon, A^{r-1}\rbrack\in\freeHLie$, and so by \eqref{nilpotenteq}, we deduce that $$B^{r-1}\lnon BA\rnon A^{r-1}\in\freeHLie .$$ Then $f_r\in\freeHLie$ for any $r$. By a similar argument and similar computations, this can be generalized into $B^xf_rA^y\in\freeHLie$ for any $x,r,y\in\Z^+$. Thus, given the relations,

\begin{eqnarray}
\lbrack B^iA^2, B^2A^k\rbrack & = &\left\{
\begin{array}{ll} 
B^{i-k+2}(B^{k-2}A^{k-2}-\algI)A^{2}, & i\geq k,\\
B^{2}(B^{i-2}A^{i-2}-\algI)A^{k-i+2}, & i< k,\end{array}\right.
\end{eqnarray}
which are consequences of \eqref{ferAnBm}, we find that $\lbrack B^iA^2, B^2A^k\rbrack\in\freeHLie$ for any $i,k\geq 4$. Then the images of $B^iA^2, B^2A^k$ under the canonical map $\freeH\rightarrow\freeH/\freeHLie$ pairwise commute. Therefore, $\freeH/\freeHLie$ is one-step nilpotent.\qed
\end{proof}

\section{The Lie algebra $\HeisenLie$ with $q$ nonzero and not a root of unity}\label{MainSec}
We assume throughout the section that $q$ is nonzero, and is not a root of unity. To remind us of this restriction on $q$, we denote $\Heisen$ and $\HeisenLie$ by $\qHeisen$ and $\qHeisenLie$, respectively.  For any $k\in\N$ and any $l\in\Z^+$, we define the following commutators in $\qHeisenLie$.
\begin{eqnarray}
\baseA(k,l) & := &  \lnon (BA)^{k}BA^{l+1}\rnon = \left(\left(\ad \lnon BA\rnon\right)^k\circ\left(-\ad A\right)^{l+1}\right)\left( B\right),\\
\baseB(k,l) & := & \lnon B^{l+1}A(BA)^{k}\rnon = \left(\left(\ad B\right)^{l-1}\circ\left(-\ad \lnon BA\rnon\right)^{k}\right)\left( \lnon B^2A\rnon\right),\\
\baseG(k) & := & \lnon B(BA)^{k}BA^{2}\rnon = \left(\left(\ad B\right)\circ\left(\ad \lnon BA\rnon\right)^{k}\right)\left( \lnon BA^2\rnon\right).
\end{eqnarray}

\begin{lemma} For any $k\in\N$ and any $l\in\Z^+$, the following relations hold in $\qHeisen$.
\begin{eqnarray}
\baseA(k,l) & = & -(1-q)^l(q^l-1)^k\LieAB^{k+1}A^l,\label{baseArel}\\
\baseB(k,l) & = & (q-1)^{k+1}(1-q^{k+1})^{l-1}B^l\LieAB^{k+1},\label{baseBrel}\\
\baseG(k) & = & q^{-k}(q-1)^{k+1}\left( q\{k\}_q\LieAB^{k+1}-\{k+1\}_q\LieAB^{k+2}\right).\label{baseGrel}
\end{eqnarray}
Furthermore, the commutators $\baseG(k)$ satisfy the following relation.
\begin{eqnarray}
q^k\sum_{i=0}^k(q-1)^{-(i+1)}\baseG(i) & = & -\{k+1\}_q\LieAB^{k+2}.\label{baseGrel2}
\end{eqnarray}
\end{lemma}
\begin{proof} Use induction on $k$ for \eqref{baseGrel}, \eqref{baseGrel2}, and induction on $k,l$ for \eqref{baseArel}, \eqref{baseBrel}.\qed
\end{proof}

Our next goal is to show that for $f,g$ arbitrarily chosen among 
\begin{eqnarray}
\lnon BA\rnon, \  A,\   B, \  \obaseA(k,l), \  \obaseB(k,l), \  \obaseG(k),\quad (k\in\N,\  l\in\Z^+),\label{candidateBasis}
\end{eqnarray}
the commutator $\lbrack f,g\rbrack$ is a linear combination of \eqref{candidateBasis}. From \eqref{baseArel}, \eqref{baseBrel}, and \eqref{baseGrel}, we have some complications in view of the scalar coefficients. For simpler computations, we define the following.
\begin{eqnarray}
\obaseA(k,l) \quad := & -(1-q)^{-l}(q^l-1)^{-k}\baseA(k,l) & = \quad \LieAB^{k+1}A^l,\label{obaseAdef}\\
\obaseB(k,l) \quad := & (q-1)^{-k-1}(1-q^{k+1})^{1-l}\baseB(k,l) & = \quad B^l\LieAB^{k+1},\label{obaseBdef}\\
\obaseG(k) \quad := & \frac{q^k}{1-q^{k+1}}\sum_{i=0}^k(q-1)^{-i}\baseG(i) & = \quad \LieAB^{k+2}.\label{obaseGdef}
\end{eqnarray}
Observe that $\obaseA(k,l)$ and $\obaseB(k,l)$ are simply scalar multiples of $\baseA(k,l)$ and $\baseB(k,l)$, respectively. Also, by some simple computation, $\LieAB^{k+2}$ can be shown to be equal to the expression to the right of $\obaseG(k)$ in \eqref{obaseGdef} using \eqref{baseGrel2}. With these simpler expressions, we only need \eqref{shiftAB}, \eqref{BnAntoLie}, \eqref{AnBntoLie} to compute for commutation relations.

\begin{proposition}\label{comrelProp1} For any $k,m\in\N$ and any $l,n\in\Z^+$, the following relations hold in $\qHeisen$.
\begin{eqnarray}
\lbrack \obaseA(k,l), \obaseA(m,n)\rbrack & = & (q^{l(m+1)}-q^{n(k+1)})\obaseA(k+m+1,l+n),\label{comrelAA}\\
\lbrack \obaseB(k,l), \obaseB(m,n) \rbrack & = & (q^{n(k+1)}-q^{l(m+1)})\obaseB(k+m+1,l+n).\label{comrelBB}
\end{eqnarray}
\end{proposition}
\begin{proof} Observe that
\begin{eqnarray}
\lbrack \obaseA(k,l), \obaseA(m,n)\rbrack & = & \lbrack\LieAB^{k+1}A^l,\LieAB^{m+1}A^n\rbrack,\nonumber\\
& = & \LieAB^{k+1}\left(A^l\LieAB^{m+1}\right)A^n - \LieAB^{m+1}\left(A^n\LieAB^{k+1}\right)A^l.\label{moveA}
\end{eqnarray}
We simply use \eqref{shiftAB} to rewrite $A^l\LieAB^{m+1}$ and $A^n\LieAB^{k+1}$ as scalar multiples of $\LieAB^{m+1}A^l$ and $\LieAB^{k+1}A^n$, respectively. The result is \eqref{comrelAA}. The relation \eqref{comrelBB} is similarly verified.\qed
\end{proof}

\begin{proposition}\label{comrelProp2} For any $k\in\N$ and any $l\in\Z^+$,
\begin{eqnarray}
\lbrack \obaseA(k,l), B\rbrack \quad = & (1-q^{-(k+2)})\{l\}_q\obaseA(k+2,l-1), & (l\geq 2),\label{comrelAB1}\\
\lbrack \obaseA(k,1), B\rbrack \quad = & q^{-(k+1)}\left(\{k+2\}_q\obaseG(k)-\{k+1\}_q\obaseG(k-1)\right), & (k\geq 1),\label{comrelAB2}\\
\lbrack \obaseA(0,1), B\rbrack \quad = & q^{-1}\left((1+q)\obaseG(0)+\lnon BA\rnon\right). &\label{comrelAB3}
\end{eqnarray}
\end{proposition}
\begin{proof} We first prove \eqref{comrelAB1}. Observe that
\begin{eqnarray}
\lbrack \obaseA(k,l), B\rbrack \quad & = & \lbrack \LieAB^{k+2}A^l,B\rbrack,\nonumber\\
& = & \LieAB^{k+2}A^lB-B\LieAB^{k+2}A^l.\label{moveBLieAB}
\end{eqnarray}
In the second term in \eqref{moveBLieAB}, we use \eqref{shiftAB} to rewrite $B\LieAB^{k+2}$ as a scalar multiple of $\LieAB^{k+2}B$. This results to
\begin{eqnarray}
\lbrack \obaseA(k,l), B\rbrack = (1-q^{-(k+2)})\LieAB^{k+2}(A^{l-1}(AB)-(BA)A^{l-1}),\nonumber
\end{eqnarray}
in which we substitute for $AB,BA$ using the following relations that are immediate from $AB-qBA=\algI$.
\begin{eqnarray}
AB & = & (q-1)^{-1}(q\LieAB - \algI),\label{ABtoLieAB}\\
BA & = & (q-1)^{-1}(\LieAB - \algI).\label{BAtoLieAB}
\end{eqnarray}
The result is 
\begin{eqnarray}
\lbrack \obaseA(k,l), B\rbrack =\frac{1-q^{-(k+2)}}{q-1}\LieAB^{k+2}(qA^{l-1}\LieAB-\LieAB A^{l-1}),\nonumber
\end{eqnarray}
on which we use \eqref{shiftAB} in order to rewrite $A^{l-1}\LieAB$ as a scalar multiple of $\LieAB A^{l-1}$. This yields \eqref{comrelAB1}. A similar pattern of computations may be made to verify \eqref{comrelAB2} and \eqref{comrelAB3}. We note in the following the important steps in the said  computations.
\begin{eqnarray}
\lbrack \obaseA(k,1), B\rbrack & = & \lbrack \LieAB^{k+1}A,B\rbrack,\nonumber\\
& = & \LieAB^{k+1}AB-B\LieAB^{k+1}A,\label{skipstep0}\\
& = & \LieAB^{k+1}(AB-q^{-(k+1)}BA),\label{skipstep1}\\
& =& q^{-(k+1)}\{k\}_q\LieAB^{k+2} - \{k+1\}_q\LieAB^{k+1},\label{skipstep2}\\
& =& q^{-(k+1)}\{k\}_q\obaseG(k) - \{k+1\}_q\LieAB^{k+1}.\label{skipstep3}
\end{eqnarray}
The transition from \eqref{skipstep0} to \eqref{skipstep1} involves the use of \eqref{shiftAB} to rewrite $B\LieAB^{k+1}$ as a scalar multiple of $\LieAB^{k+1}B$. Going from \eqref{skipstep1} to \eqref{skipstep2} involves the use of \eqref{ABtoLieAB} and \eqref{BAtoLieAB}. Finally, in \eqref{skipstep3}, the Lie monomial in the second term is either $\obaseG(k-1)$ if $k\geq 1$, or $-\lnon BA\rnon$ if $k=0$. From these we get \eqref{comrelAB2} and \eqref{comrelAB3}.\qed
\end{proof}

\begin{proposition}\label{comrelProp3} For any $m\in\N$ and any $n\in\Z^+$,
\begin{eqnarray}
\lbrack A,\obaseB(m,n)\rbrack \quad = & q^{-(m+2)}\left(\{m+n+2\}_q\obaseB(m+2,n-1)-\{m+2\}_q\obaseB(m+1,n-1)\right), & (n\geq 2),\label{comrelAB4}\\
\lbrack A,\obaseB(m,1)\rbrack \quad = & q^{-(m+1)}\left(\{m+2\}_q\obaseG(m)-\{m+1\}_q\obaseG(m-1)\right), & (m\geq 1),\label{comrelAB5}\\
\lbrack A,\obaseB(0,1)\rbrack \quad = & q^{-1}\left((1+q)\obaseG(0)+\lnon BA\rnon\right). & \label{comrelAB6}
\end{eqnarray}
\end{proposition}
\begin{proof} The computations needed to verify relations \eqref{comrelAB4} to \eqref{comrelAB6} are similar to those outlined in the proof of Proposition~\ref{comrelProp2}.\qed
\end{proof}

\begin{proposition}\label{comrelProp4} Let $k,m\in\N$ and $l,n\in\Z^+$. For any $i\in\N$, define $$c_i(k,l,m,n) := (-1)^{n-i}{{n}\choose{i}}_q \left(q^{(l-n)(i+m+1)+{{i+1}\choose{2}}}-q^{-n(k+m+\frac{n+3}{2})+{{n-i}\choose{2}}}\right).$$
Then the following relations hold in $\qHeisen$.
\begin{eqnarray}
\lbrack\obaseA(k,l),\obaseB(m,n)\rbrack = \left\{\begin{array}{ll}
(q-1)^{-n}\sum_{i=0}^nc_i(k,l,m,n)\obaseA(i+k+m+1,l-n), & l>n,\\
(q-1)^{-l}\sum_{i=0}^lc_i(m,n,k,l)\obaseB(i+k+m+1,n-l), & l<n,\\
(q-1)^{-l}\sum_{i=0}^lc_i(k,l,m,l)\obaseG(i+k+m), & l=n.
\end{array}\right.\label{BIGcomrel}
\end{eqnarray}
\end{proposition}
\begin{proof} First, we consider the case $l>n$. Observe that
\begin{eqnarray}
\lbrack\obaseA(k,l),\obaseB(m,n)\rbrack & = & \lbrack \LieAB^{k+1}A^l,B^n\LieAB^{m+1}\rbrack,\label{bigstep0}\\
& = & \LieAB^{k+1}A^lB^n\LieAB^{m+1} - B^n\LieAB^{k+m+2}A^l,\label{bigstep1}\\
& =& \LieAB^{k+m+2} \left(q^{(l-n)(m+1)}A^lB^n-q^{-n(k+m+2)}B^nA^l\right),\label{bigstep2}\\
& =& \LieAB^{k+m+2} \left(q^{(l-n)(m+1)}A^{l-n}(A^nB^n)-q^{-n(k+m+2)}(B^nA^n)A^{l-n}\right).\label{bigstep3}
\end{eqnarray}
The step from \eqref{bigstep1} to \eqref{bigstep2} is because of the use of \eqref{shiftAB} in order to rewrite $A^lB^n\LieAB^{m+1}$ and  $B^n\LieAB^{k+m+2}$ as scalar multiples of $\LieAB^{m+1}A^lB^n$ and $\LieAB^{k+m+2}B^n$, respectively. We note here that it is possible to go from \eqref{bigstep2} to \eqref{bigstep3} because of the assumption $l>n$. We use \eqref{AnBntoLie} to replace $A^nB^n$ by a linear combination of powers of $\LieAB$. The result is that the $A^{l-n}(A^nB^n)$ in \eqref{bigstep3} becomes a linear combination of vectors of the form $A^x\LieAB^y$, which can be replaced as scalar multiples of $\LieAB^yA^x$ using \eqref{shiftAB}. If we further use \eqref{BnAntoLie} to replace $B^nA^n$ in \eqref{bigstep3} as a linear combination of powers of $\LieAB$, we find that \eqref{bigstep3} is a linear combination of vectors of the form $\LieAB^uA^v$. It is routine to show that the exact scalar coefficients and exponents in each $\LieAB^uA^v$ is as given in the statement of this proposition. To prove the other cases, we need equations similar to \eqref{bigstep3}, on which we can use \eqref{shiftAB},\eqref{BnAntoLie},\eqref{AnBntoLie} in order to obtain the desired linear combinations in the second and third cases in \eqref{BIGcomrel}. The said equations are the following.
\begin{eqnarray}
\lbrack\obaseA(k,l),\obaseB(m,n)\rbrack \quad = & \left(q^{(n-l)(k+1)}(A^lB^l)B^{n-l}-q^{-l(k+m+2)}B^{n-l}(B^lA^l)\right)\LieAB^{k+m+2}, & (l<n)\label{bigstep4}\\
\lbrack\obaseA(k,l),\obaseB(m,n)\rbrack \quad = & \LieAB^{k+m+2}\left(A^lB^l-q^{-l(k+m+2)}B^lA^l\right), & (l=n).\label{bigstep5}
\end{eqnarray}
The above equations can be derived from \eqref{bigstep1} using \eqref{shiftAB}. As previously described, we replace $A^lB^l$ and $B^lA^l$ in \eqref{bigstep4}, \eqref{bigstep5} using \eqref{BnAntoLie}, \eqref{AnBntoLie}. In \eqref{bigstep5}, we get the desired relation after some routine computations involving the scalar coefficients. For \eqref{bigstep4}, we need to use \eqref{shiftAB} to rewrite any vector of the form $\LieAB^xB^y$ into a scalar multiple of  $B^y\LieAB^x$, after which it is routine to show that the exact scalar coefficients and exponents in all $B^y\LieAB^x$ are as given in the statement.\qed
\end{proof}
In Propositions~\ref{comrelProp1} to \ref{comrelProp4}, we have tackled the most complicated commutation relations. However, we still have several more cases for the commutators $\lbrack f,g\rbrack$ where $f,g$ are among $$\obaseA(k,l),\  A,\  \lnon BA\rnon,\  \obaseG(k),\  B,\  \obaseB(k,l).$$
These other cases not covered in Propositions~\ref{comrelProp1} to \ref{comrelProp4} can be obtained by simple routine computations mostly involving \eqref{shiftAB}. We summarize the results in the following table.
\begin{center}

\scalebox{0.73}{
\begin{tabular}{|c|c|c|c|c|c|c|}
\hline
 & $\obaseA(m,n)$ & $A$ & $\lnon BA\rnon$ & $\obaseG(m)$ & $B$ & $\obaseB(m,n)$ \\
\hline
$\obaseA(k,l)$ & \eqref{comrelAA} & $(1-q)\obaseA(k,l+1)$ & $(1-q^l)\obaseA(k+1,l)$ & $(q^{l(m+2)}-1)\obaseA(k+m+2,l)$ & \eqref{comrelAB1},\eqref{comrelAB2},\eqref{comrelAB3} & \eqref{BIGcomrel}\\
\hline
$A$ & & $0$ & $(1-q)\obaseA(0,1)$ & $(q^{m+2}-1)\obaseA(m+1,1)$ & $-\lnon BA\rnon$ & \eqref{comrelAB4}\\
\hline
$\lnon BA\rnon$ & & & $0$ & $0$ & $(1-q)\obaseB(0,1)$ & $(1-q^n)\obaseB(m+1,n)$\\
\hline
$\obaseG(k)$ & & & & $0$ & $(q^{k+2}-1)\obaseB(k+1,1)$ & $(q^{n(k+2)}-1)\obaseB(k+m+2,n)$\\
\hline
$B$ & & & & & $0$ & $(1-q^{m+1})\obaseB(m,n+1)$\\
\hline
$\obaseB(k,l)$ & & & & & & \eqref{comrelBB}\\
\hline
\end{tabular}}
\captionof{table}{Commutation relations for the vectors $\  \obaseA(k,l),\  A,\  \lnon BA\rnon,\  \obaseG(k),\  B,\  \obaseB(k,l).$}\label{table:COMtable}
\end{center}
For any $f,g$ among $\  \obaseA(k,l),\  A,\  \lnon BA\rnon,\  \obaseG(k),\  B,\  \obaseB(k,l)$, all the possibilities for $f$ are listed in the first column, while all possibilities for $g$ are listed in the first row. The commutator $\lbrack f,g\rbrack$ can be found on the corresponding cell. Because of the Lie algebra property $\lbrack v,u\rbrack = - \lbrack u,v\rbrack$ for any $u,v\in\qHeisenLie$, some cells in the table are left blank. Our most important observation here is the following. 

\begin{remark}\label{closeRem}\emph{ With reference to Table~\ref{table:COMtable}, the span of  $\obaseA(k,l),\  A,\  \lnon BA\rnon,\  \obaseG(k),\  B,\  \obaseB(k,l)$ is closed under the Lie bracket.
}\end{remark}

\begin{lemma}\label{grandindepLem} The following vectors form a basis for $\qHeisen$.
\begin{eqnarray}
\lnon BA\rnon, \  A^k,\   B^l, \  \obaseA(k,l), \  \obaseB(k,l), \  \obaseG(k),\quad (k\in\N,\  l\in\Z^+).\label{semigrandBasis}
\end{eqnarray}
\end{lemma}
\begin{proof} This immediately follows from the fact that $\lnon BA\rnon = -\LieAB$, the equations for $\obaseA(k,l), \  \obaseB(k,l), \  \obaseG(k)$ in \eqref{obaseAdef} to \eqref{obaseGdef}, and Proposition~\ref{gradProp}.\qed
\end{proof}

\begin{theorem}\label{grandBasisThm} The following vectors form a basis for $\qHeisenLie$. 
\begin{eqnarray}
\lnon BA\rnon, \  A,\   B, \  \obaseA(k,l), \  \obaseB(k,l), \  \obaseG(k),\quad (k\in\N,\  l\in\Z^+).\label{grandBasis}
\end{eqnarray}
The commutation relations for the above basis are given in Table~\ref{table:COMtable}.
\end{theorem}
\begin{proof} The proof that \eqref{grandBasis} is a basis for $\qHeisenLie$ is similar to that in the proof of Theorem~\ref{freeThm}. Linear independence follows from  Lemma~\ref{grandindepLem}, and also observe that all vectors in \eqref{grandBasis} are in $\qHeisenLie$. Thus, to show spanning, we prove that all nonassociative regular words on $A<B$, which form a spanning set for $\qHeisenLie$, belong to the span of \eqref{grandBasis}. To do this, given a regular word $W$, we use induction on $\BActr(W)$. Observe that $\lnon BA^n\rnon$ is either $\lnon BA\rnon$ or $\baseA(0,n-1)$, the latter being a scalar multiple of $\obaseA(0,n-1)$. Thus, in any case, $\lnon BA^n\rnon$ is in the span of \eqref{grandBasis}. By Remark~\ref{closeRem} and induction on $m$, it is routine to show $\lnon B^mA^n\rnon$ is in the span of \eqref{grandBasis} for all $m$. This serves as proof for $\BActr(W)=1$. Using Remark~\ref{closeRem} again, the induction step is similar to that done in the proof of Theorem~\ref{freeThm}.\qed
\end{proof}

\begin{lemma} Let $k\in\N$ and $l,m,n\in\Z^+$ with $m,n\geq 2$. Then the following hold.
\begin{eqnarray}
\lbrack A^n,\obaseB(k,l)\rbrack & \in & \qHeisenLie,\label{ideal0}\\
\lbrack B^m,\obaseA(k,l)\rbrack & \in & \qHeisenLie.\label{ideal1}
\end{eqnarray}
\end{lemma}
\begin{proof} Define $f:=A^nB^l-q^{-n(k+1)}B^lA^n\in\filtrH_{l-n}$. By some routine calculations, $\lbrack A^n,\obaseB(k,l)\rbrack = f\LieAB^{k+1}$. If $l=n$, then by Proposition~\ref{gradProp}, $f$ is a linear combination of $\algI,\lnon BA\rnon, \obaseG(t)$ for $0\leq t\leq s$ for some $s$. Since $k+1\geq 1$, we find that $f\LieAB^{k+1}$ is a linear combination of the $\lnon BA\rnon, \obaseG(t)$, and so $f\LieAB^{k+1}\in\qHeisenLie$. If $l>n$ or $l<n$, by similar reasoning, we find that $f\LieAB^{k+1}$ is a linear combination of vectors of the form $\obaseA(x,y)$ or $\obaseB(x,y)$, and so $f\LieAB^{k+1}\in\qHeisenLie$. The proof for \eqref{ideal1} is similar.\qed
\end{proof}

\begin{theorem} The Lie subalgebra $\qHeisenLie$ of $\qHeisen$ is a Lie ideal of $\qHeisen$, and the resulting quotient Lie algebra $\barQuotient$ is infinite-dimensional and one-step nilpotent.
\end{theorem}
\begin{proof} By Lemma~\ref{grandindepLem} and Theorem~\ref{grandBasisThm}, it suffices to show that $\qHeisenLie$ contains all Lie products of the form $\lbrack L,R\rbrack$ where $L$ is one of the vectors in \eqref{semigrandBasis} not in \eqref{grandBasis}, and $R$ is one of the vectors in \eqref{grandBasis}. To this end, we summarize relevant information about all such Lie products $\lbrack L,R\rbrack$ in the following table. Observe that $L\in\{ A^n, B^m\  |\  m,n\geq 2\}$ and these are listed in the first column. The possibilities for $R$ are listed in the first row.

\begin{center}
\scalebox{1.0}{
\begin{tabular}{|c|c|c|c|c|}
\hline
 &  $\lnon BA\rnon$ & $\obaseA(k,l)$ & $\obaseB(k,l)$ & $\obaseG(k)$ \\
\hline
$A^n$  & $(1-q^n)\obaseA(0,n)$ & $(q^{n(k+1)}-1)\obaseA(k,l+m)$ & \eqref{ideal0} & $(q^{n(k+2)}-1)\obaseA(k+1,n)$\\
\hline
$B^m$  & $(q^m-1)\obaseB(0,m)$ & \eqref{ideal1}& $(1-q^{m(k+1)})\obaseB(k,l+m)$ & $(1-q^{m(k+2)})\obaseB(k+1,m)$\\
\hline
\end{tabular}}
\captionof{table}{Lie products for $\qHeisenLie$ as a Lie ideal of $\qHeisen$}\label{table:IDEALtable}
\end{center}
Therefore, $\qHeisenLie$ is a Lie ideal of $\qHeisen$. By Lemma~\ref{grandindepLem} and Theorem~\ref{grandBasisThm}, a basis for $\barQuotient$ consists of the images of $A^n$, $B^m$, $m,n\geq 0$ under the canonical Lie algebra homomorphism $\qHeisen\rightarrow\barQuotient$. Thus, $\barQuotient$ is infinite-dimensional. To show $\barQuotient$ is one-step nilpotent, it suffices to show that in $\qHeisen$, the Lie product $\lbrack B^m,A^n\rbrack$ (for any $ m,n\geq 2$) is an element of $\qHeisenLie$. By routine computations it can be shown that $\lbrack B^m,A^n\rbrack$ is an element of the span of 
\begin{eqnarray}
\{\lnon BA\rnon,\  \obaseB(i-1,m-n) & | & i\in\N\},\nonumber\\
\{\obaseA(i-1,n-m)  & | & i\in\N\},\nonumber\\
\{\lnon BA\rnon,\  \obaseG(k) & | &  k\in\N\},\nonumber
\end{eqnarray}
for the cases $m>n$, $m<n$, and $m=n$, respectively. This completes the proof.\qed
\end{proof}


\section{Further directions}

We invite the interested reader to explore the case when $q$ is a root of unity. Also, we note that investigation on commutators in associative algebras with $q$-deformed commutation relations is a relatively new territory of exploration. The idea of exploring the interplay between a presentation with $q$-deformed commutation relations and the non-deformed commutators in the same algebraic structure was proposed in \cite[Problem 12.14]{UAW}, but this was for an algebra that is not $\Heisen$, and has more complicated $q$-deformed defining relations. In \cite{Can}, it was shown that there exists an associative algebra whose defining relations are $q$-deformed commutation relations that are not Lie polynomials, but the underlying Lie subalgebra is not free. Based on the state of current literature, if any, about our topic, we consider as open the question of whether a Lie algebra presentation can be deduced from an associative algebra presentation whose relations are $q$-deformed commutation relations that do not consist of Lie polynomials.




\begin{thebibliography}{35}

\bibitem{Berg} G. Bergman, The diamond lemma for ring theory, {\it Adv. Math.} {\bf 29} (1978) 178-218.

\bibitem{Can} R. Cantuba, A Lie algebra related to the universal Askey-Wilson algebra, Matimy\'{a}s Matematika, {\bf 38} (2015) 51-75, available at \url{http://mathsociety.ph/matimyas/images/vol38/Cantuba2.pdf}.

\bibitem{Cart} R. Carter, Lie algebras of finite and affine type, {\it Cambridge Studies in Advanced Mathematics,} Vol. 96, Cambridge University Press, Cambridge, 2005.

\bibitem{Hall} M. Hall, A basis for free Lie rings and higher commutators in free groups. {\it Proc. Amer. Math. Soc.} {\bf 1} (1950) 575-581.

\bibitem{Hel} L. Hellstr\"{o}m, S. Silvestrov, Commuting elements in $q$-deformed Heisenberg algebras. {\it World Scientific}, 2000.

\bibitem{Hel05} L. Hellstr\"{o}m, S. Silvestrov, Two-sided ideals in $q$-deformed Heisenberg algebras. {\it Expo. Math.} {\bf 23} (2005) 99-125.

\bibitem{Hor} A. Hora, N. Obata, Quantum probability and spectral analysis of graphs, {\it Springer-Verlag}, 2007.

\bibitem{Reut} C. Reutenauer, Free Lie algebras, {\it Oxford Univ. Press, New York,} 1993.

\bibitem{Shir53} Shirshov, A., Subalgebras of free Lie algebras, {\it Mat. Sbornik N.S.}, {\bf 33} (1953) 441–452.

\bibitem{Shir58} Shirshov, A., On free Lie rings, {\it Mat. Sbornik N.S.}, {\bf 45} (1958) 113-122.

\bibitem{UAW} P. Terwilliger, The universal Askey-Wilson algebra, {\it SIGMA} {\bf 7} (2011) 069, 24 pages, arXiv:1104.2813.

\bibitem{Ufn} V. Ufnarovskij, Combinatorial and asymptotic methods in algebra, {\it Encyclopedia of Mathematical Sciences,} Vol. 57, Springer, 1995.

\end{thebibliography}
\end{document}